\documentclass[11pt]{amsart}
\usepackage{mathtools}
\usepackage{amsmath}
\usepackage{amssymb}
\usepackage{amsthm}
\usepackage{yhmath}
\usepackage{graphicx}
\usepackage[colorlinks,citecolor=cyan,linkcolor=magenta]{hyperref}
\usepackage{ mathrsfs }
\usepackage{bbm}
\usepackage{xcolor}
\usepackage{tikz-cd}
\usepackage{cleveref}
\usepackage{verbatim}
\usepackage{enumitem}
\usepackage{scalerel}
\usepackage{wrapfig}
\usepackage{float}
\usepackage{manfnt}
\newtheorem{theorem}{Theorem}[section]

\newtheorem*{theorem*}{Theorem}
\newtheorem*{definition*}{Definition}
\newtheorem*{prop*}{Proposition}
\newtheorem*{cor*}{Corollary}
\newtheorem*{lemma*}{Lemma}

\newtheorem*{claim*}{Claim}

\newtheorem{lemma}[theorem]{Lemma}
\newtheorem{lem}[theorem]{Lemma}

\newtheorem{cor}[theorem]{Corollary}

\newtheorem{prop}[theorem]{Proposition}

\newtheorem{thm}[theorem]{Theorem}

\newtheorem{Fact}[theorem]{Fact}

\theoremstyle{definition}
\newtheorem{definition}[theorem]{Definition}

\theoremstyle{remark}
\newtheorem{remark}[theorem]{Remark}
\newtheorem*{remarks*}{Remarks}

\numberwithin{equation}{section}

\newcommand{\Z}{\mathbb{Z}}

\newcommand{\Hplane}{\mathbb{H}^2}

\newcommand{\op}{\operatorname}

\newcommand{\T}{\op{T}}

\newcommand{\be}{\begin{equation}}
\newcommand{\ee}{\end{equation}}

\renewcommand{\T}{\mathsf{T}}

\newcommand{\N}{\mathbb N}

\renewcommand{\epsilon}{\varepsilon}

\newcommand{\R}{\mathbb{R}}

\newcommand{\PSL}{\op{PSL}}

\newcommand{\GmodGamma}{G / \Gamma}

\newcommand{\s}{\mathscr{S}_+}

\newcommand{\Imp}{\mathrm{Im}}
\newcommand{\Rep}{\mathrm{Re}}

\begin{document}
	
	\graphicspath{ {} }
	
	\title{Weaving Geodesics and New Phenomena in Horocyclic Dynamics}
 	\author{Fran\c{c}oise Dal'bo}
 	\address{Institut de recherche mathématique de Rennes (UMR CNRS 6625)\\ French-Uruguayan Laboratory IFUMI (IRL CNRS 2030)}
 	\email{francoise.dalbo@univ-rennes.fr}
 	\author{James Farre}
 	\address{Max Planck Institute for mathematics in the sciences (MPI Mis) in Leipzig}
 	\email{james.farre@mis.mpg.de}
    \author{Or Landesberg}
    \address{Department of Mathematics, Hebrew University of Jerusalem}
    \email{or.landesberg@mail.huji.ac.il}
    \author{Yair Minsky}
    \address{Department of Mathematics, Yale University}
    \email{yair.minsky@yale.edu}
	
	\begin{abstract}
		We construct geometrically infinite hyperbolic surfaces supporting horocycles with tailored recurrence properties. In particular, we obtain the first examples of non-trivial minimal horocyclic orbit closures and of infinite locally-finite conservative horocyclic invariant measures which are singular with respect to the geodesic flow.
		Other examples include surfaces supporting horocyclic orbit closures of arbitrary Hausdorff dimension in $(1,2)$. 
	\end{abstract}

    \vspace*{-0.8cm}
	\maketitle

	The study of horospherical flows on hyperbolic manifolds dates back to Hedlund in the 1930s \cite{hedlundFuchsianGroupsTransitive1936} and has played an important role in the development of modern homogeneous dynamics. In the finite-volume (and geometrically finite) cases, both the measure-theoretic and topological properties of the flow have been extensively studied, revealing a remarkable degree of rigidity (see, e.g.~\cite{furstenbergUniqueErgodicityHorocycle1973,daniUniformDistributionHorocycle1984,burgerHorocycleFlowGeometrically1990,roblinErgodiciteEquidistributionCourbure2003,ratnerRaghunathansMeasureConjecture1991}). In contrast, the behavior of horospherical flows in the general geometrically infinite setting remains much less well understood.

	Until recently the only explicitly described examples of horocyclic orbit closures on orientable hyperbolic surfaces were ``trivial'' --- either the full non-wandering set for the horocyclic flow or single closed horocycles. While the existence of other, more intricate, orbit closures was known for decades, none have been described in detail, leaving much mystery as to their potential regularity and rigidity properties (c.f.~\cite{dalboClassificationLimitPoints2000,coudeneHorocyclesRecurrentsSurfaces2010a,gayeLinexistenceDensemblesMinimaux2017,matsumotoHorocycleFlowsMinimal2016,bellisLinksHorocyclicGeodesic2018,ledrappierHorospheresAbelianCovers1997,ledrappierErratumHorospheresAbelian1998}).
	
	In recent works \cite{farreMinimizingLaminationsRegular2023a,farreClassificationHorocycleOrbit2024}, an explicit description of all horocyclic orbit closures was given in the setting of $\Z$-covers of compact hyperbolic surfaces. These orbit closures were shown to be highly irregular; their structural features depend in a delicate way on the geometry of the underlying surface.  Intriguingly, all had integer Hausdorff dimension. They are also non-minimal. 
    
    Notably, no non-trivial minimal subsets for the horocyclic flow were known before now.
	\smallskip
	
	In this paper, we provide the first examples of non-trivial minimal horocyclic subsets as well as new fractional dimensional orbit closures. 
	
	As a consequence of our constructions, we provide the first counterexamples to the horospherical infinite measure rigidity phenomenon, which has been observed in a vast variety of settings, where every horospherically invariant ergodic Radon measure is either quasi-invariant under the geodesic flow (à la Babillot–Ledrappier) or supported on a single closed orbit; see \cite{burgerHorocycleFlowGeometrically1990,roblinErgodiciteEquidistributionCourbure2003,sarigInvariantRadonMeasures2004,ledrappierInvariantMeasuresHorocycle2007,sarigHorocyclicFlowLaplacian2010,ohLocalMixingInvariant2017,sarigHorocycleFlowsSurfaces2019,landesbergRadonMeasuresInvariant2022,landesbergHorosphericallyInvariantMeasures2021,landesbergHorosphericalInvariantMeasures2023a} and \Cref{rem:Patterson example}.
	
	\subsection*{Main Results} Let $\Sigma$ be any orientable hyperbolic surface with unit tangent bundle $\T^1\Sigma \cong G/\Gamma$, where $G = \PSL_2(\R)$ and $\Gamma\le G$ is a discrete torsion-free subgroup acting isometrically on the right.
	Let $A=\{a_t=\mathrm{diag}(e^{t/2},e^{-t/2})\}_{t \in \R}$ denote the diagonal subgroup of $G$ generating, via left multiplication, the geodesic flow.  Let $A_+ = \{a_t : t\ge 0\}$, and let $N\le G$ be the lower unipotent subgroup corresponding to the stable horocyclic flow on $\T^1\Sigma$. We denote by $p$ the projection map from the unit-tangent bundle (of either $\Hplane$ or $\Sigma$) down to the surface.
	
	Given a discrete subgroup $\Gamma \leq G$, we denote by $\Lambda \subseteq \partial \Hplane$ its limit set. The non-wandering set for the horocycle flow is
	\[ \mathcal{E} = \{g\Gamma \in \GmodGamma : g^+ \in \Lambda\}, \]
	where $g^+$ is the terminal endpoint in $\partial \Hplane$ of the geodesic ray emanating from $g$.
	
	Recall that a non-empty $N$-invariant closed set $F \subseteq \Sigma$ is called $N$-minimal if all $N$-orbits in $F$ are dense in $F$. A characterization of points with dense horocyclic orbits in $\mathcal E$ is given by \cite{eberleinHorocycleFlowsCertain1977, dalboTopologieFeuilletageFortement2000} where it was shown that $\overline{Nx}\neq\mathcal E$ if and only if the geodesic ray $A_+x$ is quasi-minimizing, that is, $d_{\T^1\Sigma}(a_tx, x) \ge t-c$ for some $c\geq 0$ and every $t\geq 0$. 
	As a consequence, $\mathcal E$ is $N$-minimal if and only if $\Gamma$ is convex co-compact.

	Studying the different possible trajectories of quasi-minimizing rays and their ``efficiency'' has turned out to be key in the analysis of horocyclic orbit closures. 
	Drawing on techniques developed in \cite{farreMinimizingLaminationsRegular2023a,farreClassificationHorocycleOrbit2024} and inspired by examples introduced by Alexandre Bellis in \cite[\S 1.5.1]{bellisEtudeTopologiqueFlot2018}, we provide a recipe for tailoring geometrically infinite surfaces supporting horocycles with prescribed recurrence properties.
	Our main results are the following:
	\begin{theorem*}\;
		\begin{enumerate}[leftmargin=*]
			\item There exists a surface $\Sigma$ such that $\T^1\Sigma$ supports an $N$-minimal closed subset which is neither $\mathcal{E}$ nor a single $N$-orbit. Moreover, this minimal orbit closure supports an $N$-invariant, ergodic, infinite and locally finite measure $\mu$ which is conservative but singular with respect to the geodesic flow, that is, $a_t.\mu\perp \mu$ for all $t \neq 0$.
			\item For any $\alpha \in (1,2)$ there exists a surface $\Sigma_\alpha$ such that $\T^1\Sigma_\alpha$ supports an $\alpha$-Hausdorff dimensional horocyclic orbit closure.
		\end{enumerate}
	\end{theorem*}

	\begin{remarks*}
		\begin{itemize}[leftmargin=*]
			\item Our surfaces are extremely sparse; the injectivity radius along all diverging geodesic rays tends to infinity. This implies, in particular, that the tameness conditions imposed in \cite{sarigHorocyclicFlowLaplacian2010} to deduce measure rigidity cannot be removed. Equivalent geometric conditions appear in \cite{landesbergRadonMeasuresInvariant2022}.
			\item The possible non-regularity of orbit closures we construct is quite extreme, allowing us to construct orbit closures having disagreeing Hausdorff and lower/upper Minkowski dimensions; see \S\ref{subsec:varying dimensions}.
			\item Note that for any $\lambda \in (0,1]$ there exist convex co-compact Fuchsian groups having $\lambda$-dimensional limit sets. In such surfaces, the corresponding orbit closure $\mathcal{E}$, being $AN$-invariant, is hence $2+\lambda$-dimensional. We may thus conclude that any $\alpha \in [1,3]$ can be the dimension of some horocycle orbit closure.
		\end{itemize}
	\end{remarks*}

	\subsubsection*{Remark to the reader about the proof}
		While we rely on techniques developed in \cite{farreMinimizingLaminationsRegular2023a,farreClassificationHorocycleOrbit2024}, we will only make use of several elementary insights and lemmas from said papers. Our proof is fairly self-contained and requires no prior knowledge or understanding of the results in the $\Z$-cover setting.

	\section{Setup}

	\subsection{Loom Surfaces}
	
	It will be convenient for us to work with the \emph{band model} for the hyperbolic plane, that is, the space $\Hplane:=\{z \in \mathbb{C} : |\Imp z|<\pi/2 \}$ equipped with the metric $|dz|/\cos \Imp z$. 

	Given a closed convex domain $J \subset \Hplane$ with totally geodesic boundary we denote by $\hat{J}$ its double, that is, the space 
	\[ \widehat{J} = \overline{J}\times \{0,1\} / \sim \qquad \text{where} \quad (z,0)\sim (z,1) \quad \text{for all } z \in \partial J. \]
	Under these conditions, $\widehat{J}$ is a complete hyperbolic surface without boundary.
	
	For $s \in \R$ and $h \in (0,\pi/2)$ we denote by $D_h(s)$ the unique open half-plane contained in $\{\Imp z > 0 \} \cap \Hplane $ and bounded by the geodesic which is perpendicular to $s+(-\pi/2,\pi/2)i$ at the point $s+hi$, see \Cref{fig: band model}. 
	
	\begin{figure}[h]
		\centering
		\includegraphics[width=0.92\linewidth]{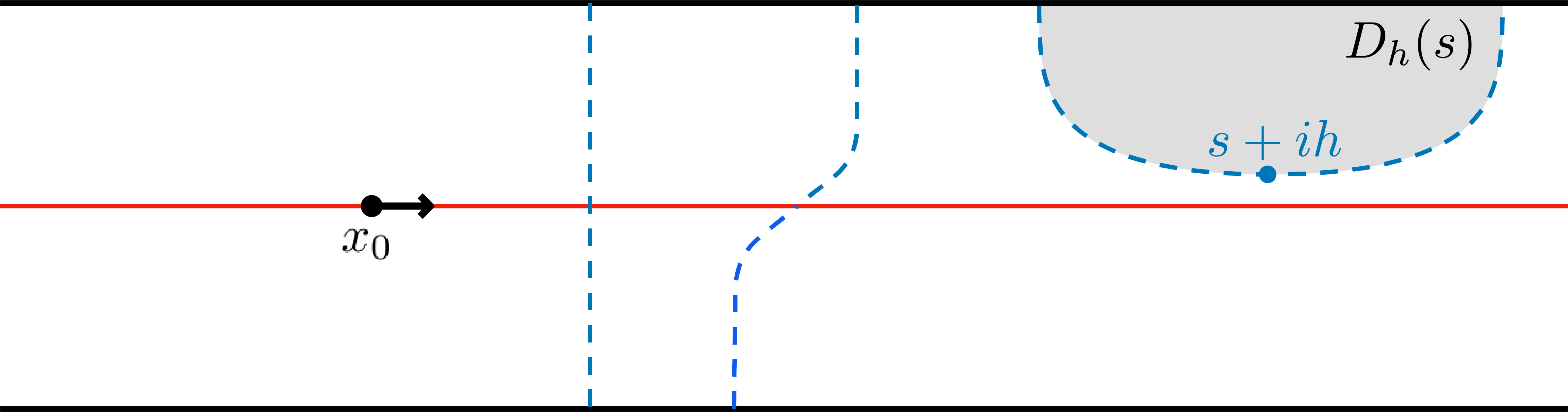}
		\caption{\small Different geodesics in the band model (in blue and red). A half plane $D_h(s)$ (shaded in gray).}
		\label{fig: band model}
	\end{figure}
	
	Given a sequence $s=(s_j,h_j)_{j \in \N}$ satisfying $\overline{D_{h_k}(s_k)} \cap \overline{D_{h_j}(s_j)} = \emptyset$ for all $k\neq j$, we consider the surface
	\[ \Sigma_s=\widehat{J_s} \quad \text{where} \quad J_s = \Hplane \smallsetminus \bigcup_k D_{h_k}(s_k). \]
    Let $q: \Sigma_s \to \overline{J_s} \subset \mathbb H^2$ be the quotient mapping identifying the two copies of $J_s$ comprising $\Sigma_s$.   
	Topologically, $\Sigma_s$ is a plane with a countable discrete set of punctures, see \Cref{fig: Birdseye view}.
	
	\begin{definition}
		A \emph{loom} surface is a surface $\Sigma_s$ as above, where the sequence $s=(s_k,h_k)_{k \in \N}$ has $s_k$ monotonic increasing, $h_k$ bounded above by $c<\pi/2$, and satisfying
		\[ d_{\Hplane}(\partial D_{h_k}(s_k), \partial D_{h_{k+1}}(s_{k+1})) \to \infty \quad \text{where} \quad k \to \infty. \]
	\end{definition}
    
	Note that for any sequence $0<h_k$ uniformly bounded away from $\frac{\pi}{2}$ there exist $s_k \to \infty$, sufficiently large and spread apart, for which $s=(s_k,h_k)_{k\in \N}$ defines a loom surface.
	
	Under the above conditions the function $\text{Inj--rad}: \Sigma_s \to (0,\infty)$, assigning the injectivity radius at a point, is a proper map. In particular, injectivity radius tends to infinity along any diverging geodesic ray.
	
	\begin{figure}[h]
		\centering
		\includegraphics[width=.9\linewidth]{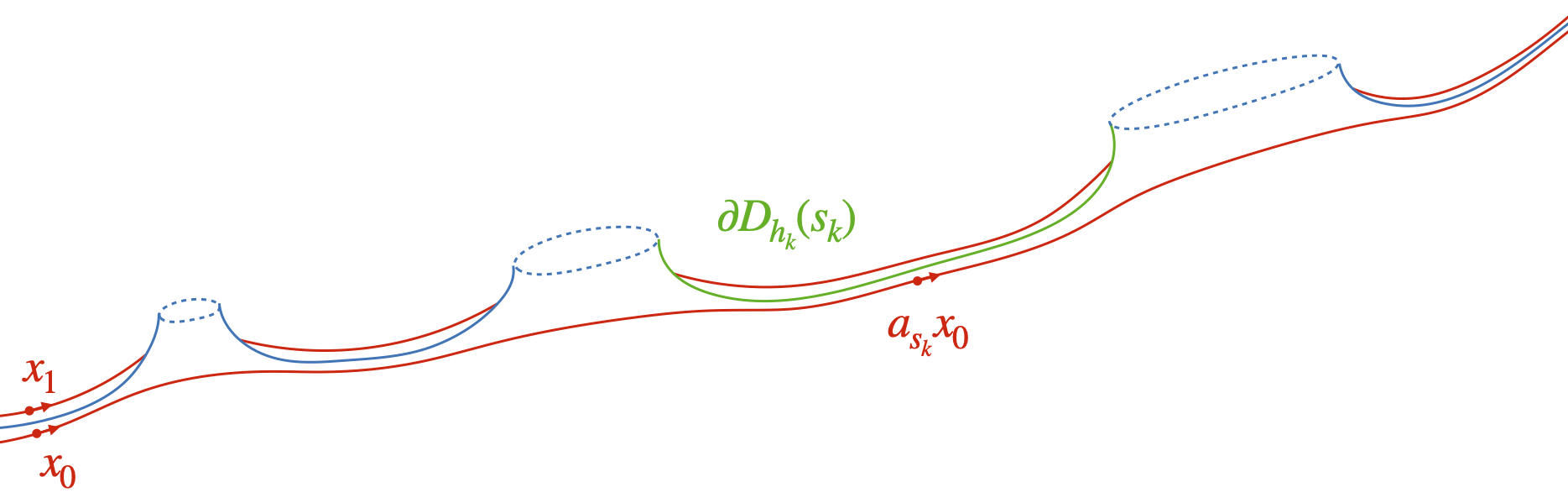}
		\caption{\small }
		\label{fig: Birdseye view}
	\end{figure}
	
	\subsection{Tight map and stretch lamination}
	
	Consider $\Sigma_s$ as above and denote by $x_j$, $j=0,1$, the points in $\T^1\Sigma_s$ corresponding to $(0,j) \in J_s\times\{j\}$ with horizontal unit vector $1 \in S^1$.
    Let $\ell\subset J_s \subset \mathbb H^2$ be the image of $p(Ax_0)$ or $p(Ax_1)$ under the quotient mapping $q$.  
    We identify $\R \cong \ell$ by the rule $t \mapsto q(p(a_tx_j))$, and define $\tau: \Sigma_s \to \mathbb R$ as the composition of $q$ followed by the nearest point projection to $\ell$.\footnote{The nearest point projection to $\ell$ in $J_s$ is the restriction of the nearest point projection to $\ell$ in $\mathbb H^2$, as $\ell\subset J_s$ and $J_s$ is geodesically convex. This map is $1$-Lipschitz and a strict contraction away from $\ell$.}

    As the composition of $1$-Lipschitz maps, $\tau$ is itself $1$-Lipschitz.  Let $\lambda = p(Ax_0) \cup p(Ax_1)$ and $\T_+^1\lambda = Ax_0 \cup Ax_1$. Observe that $\tau$ is  strictly contracting away from $\lambda$ and is isometric along each of its components.  In particular, $p(Ax_0)$ and $p(Ax_1)$ are isometrically embedded in $\Sigma$, and $Ax_0$ and $Ax_1$ are isometrically embedded in $\T^1\Sigma$.
      
     Abusing notation, we also use $\tau$ to denote $p^*\tau : \T^1\Sigma_s \to \R$, which is constant along fibers of $p$.  For $y \in \T^1\Sigma_s$,  $\tau(y) =t$ means that the closest point to $p(y)$ in $\lambda$ is $p(a_tx_0)$ or $p(a_tx_1)$.  Equivalently,  $\tau(y) = \Rep(q(p(y))$, where $\Rep(z)$ is the real part of the complex number $z$.	

	\subsection{Slack}
	
	We consider the notion of $\tau$-slack introduced in \cite[\S3]{farreClassificationHorocycleOrbit2024} which measures how ``efficiently'' a path progresses towards the end at $+\infty$:
	
	\begin{definition}\label{def: slack}
		
		Let $ \alpha: [a,b] \to \Sigma_s $ be a rectifiable curve. We define the \emph{slack of $ \alpha $} to be
		\[ \s(\alpha) = \mathrm{length}(\alpha) -
		(\tau(\alpha(b))-\tau(\alpha(a))). \]
		Similarly if $\beta:[a,b]\to\T^1\Sigma_s$ is rectifiable we define its slack to be the slack of its projection to $\Sigma_s$. 
		Note that $\s$ is non-negative and additive under concatenation of paths, so if $I\subset \R$ is connected and $\alpha: I \to \T^1\Sigma_s $, we can define \[\s(\alpha) = \lim_{T \to \infty} \s\left(\alpha|_{I\cap [-T,T]}\right)
		=\sup_{T >0} \s\left(\alpha|_{I\cap [-T,T]}\right)\in [0,\infty].\] 
	\end{definition}
	
	If $\alpha$ is a geodesic flow line, of the form $A_{[s,t]}z$, we note that $\s(\alpha)$
	is just $(t-s) - (\tau(a_t z)-\tau(a_s z))$ and that
	\begin{equation}\label{eqn: zero slack in lambda}
		\s(\alpha) = 0 \text{ if and only if } \alpha\subset \T^1_+\lambda.  
	\end{equation}
	
	\subsection{Busemann-type Function} 
	Consider the function $\beta: \T^1\Sigma_s \to [-\infty,\infty)$ defined by
	\begin{equation}\label{eq:Definition of beta}
		\beta(y)=\tau(y)-\s(A_+y) = \lim_{t \to +\infty} \tau(a_ty)-t.
	\end{equation}
	It is upper semi-continuous, as a decreasing limit of continuous functions, and also $N$-invariant, see \cite[Lemma 6.1]{farreMinimizingLaminationsRegular2023a}. Therefore, for all $y \in \T^1 \Sigma_s$ we have
	\begin{equation}\label{eq:beta horoball}
		\overline{Ny}\subseteq \beta^{-1}([\beta(y),\infty)).
	\end{equation} 
	
	In particular, since $\tau(x_j)=0$ and $\s(A_+x_j)=0$ for both $j=0,1$, we conclude:
	\begin{Fact}\label{Fact:beta horoball supset}
		For both $j=0,1$, all $y \in \overline{Nx_j}$ satisfy $\beta(y)\geq 0$.
	\end{Fact}
	
	\subsection{Weaving Lemma}
	A fundamental observation is the following: any geodesic trajectory $A_+y$ spending an infinite amount of time a definite distance away from the isometric locus of $\tau$, $\T^1_+\lambda$, will necessarily have $\beta(y)=-\infty$. In other words,
	\begin{Fact}\label{Fact:asymp_to_lambda}
		All $y \in \T^1\Sigma_s$ with $\beta(y) > -\infty$ satisfy that for all $\varepsilon >0$ there exists $T>0$ such that $d(a_ty,\T^1_+\lambda) < \varepsilon$ for all $t > T$.
	\end{Fact} 
	\noindent See the first part of the proof of \cite[Thm. 3.4]{farreMinimizingLaminationsRegular2023a} for more details.
	
	In the case of loom surfaces, this asymptotic behavior ensures all finite slack geodesic rays eventually follow some \emph{weaving pattern}. 
	
	\begin{definition}\label{def:weaving geodesic ray}
		A geodesic ray $A_+w$ is said to be \emph{weaving} with \emph{pattern} $W=\{k_1 < k_2 <...\} \subseteq \N $ if $t\mapsto \tau (a_tw)$ is monotonically increasing and unbounded, for $t\geq 0$, and
		\[ W=\left\{k \in \N : p(A_+w)\cap \partial D_{h_k}(s_k)\neq \emptyset \right\}. \]
	\end{definition}
	
	\begin{lem}[Weaving Lemma]\label{lemma:weaving}
		In a loom surface $\Sigma_s$, weaving geodesics are uniquely determined by their weaving pattern and initial point.
		
		In addition, for any $ \rho \geq 0 $ there exists $S>0$ such that any geodesic ray $A_+y $ with slack $\s(A_+y) \leq \rho$ and beginning at $\tau(y)>S$ is weaving.
	\end{lem}
	
	\begin{proof}
		First note that since both $J_s$ copies are convex and simply connected, any geodesic ray passing from $J_s\times \{j\}$ through $\partial D_{h_k}(s_k)$ to $J_s\times\{j+1 \!\!\mod 2\}$ will either stay indefinitely in that side or have its first return to $J_s\times \{j\}$ happen through a different component of $\partial J_s$. 
		Whenever a weaving geodesic has an infinite weaving pattern, monotonicity of $\tau$-values ensures the indices in the pattern appear in strictly increasing order. Uniqueness of geodesic representatives in each homotopy class relative endpoints implies such a weaving geodesic is uniquely determined.
		
		On the other hand, any geodesic ray entirely contained within one $J_s\times \{j\}$ and having unbounded $\tau$-values is necessarily asymptotic to $A_+x_j$. We may thus deduce the claim for weaving geodesics with a finite pattern as well.
		\medskip
		
		For the second claim, given any $\rho \geq 0$ there exists $k_0 \in \N$ such		
		\[ \rho < \min_{k \geq k_0} d_{\Hplane}(\partial D_{h_k}(s_k), \partial D_{h_{k-1}}(s_{k-1})). \]
		Accordingly, choose $S=s_{k_0}$. Any path beginning at $\tau(y)>S$ and passing through $\partial D_{h_k}(s_k)$ for $k<k_0$ must have a segment of length$>\!\rho$ which is in the ``wrong'' direction and therefore must have slack bigger than $\rho$. 
        Similarly, for $k \ge k_0$, any path connecting $\partial D_{h_k}(s_k)$ with a boundary component of smaller index also has slack at least $\rho$, proving the claim.
	\end{proof}
	 A direct consequence of \Cref{lemma:weaving} is that any finite slack geodesic ray is eventually weaving.	
	\medskip
	
	Given $k \in \N$ we denote by $\eta_k^+$ the unique geodesic in $\Sigma_s$ which is backward asymptotic to $A_-x_0$ and forward asymptotic to $A_+x_1$ crossing once from $J_s\times \{0\}$ to $J_s\times \{1\}$ through $\partial D_{h_k}(s_k)$. We call $\eta_k^+$ a \emph{crossing}. 
	We similarly have $\eta_k^-$, the symmetric geodesic passing from $J_s\times \{1\}$ to $J_s\times \{0\}$ through $\partial D_{h_k}(s_k)$, see \Cref{fig: Crossings}.
	Symmetry of the construction implies $\eta_k^\pm \cap \partial D_{h_k}(s_k) = \{s_k+ih_k\}$.
	
	\begin{figure}[h]
		\centering
		\includegraphics[width=1\linewidth]{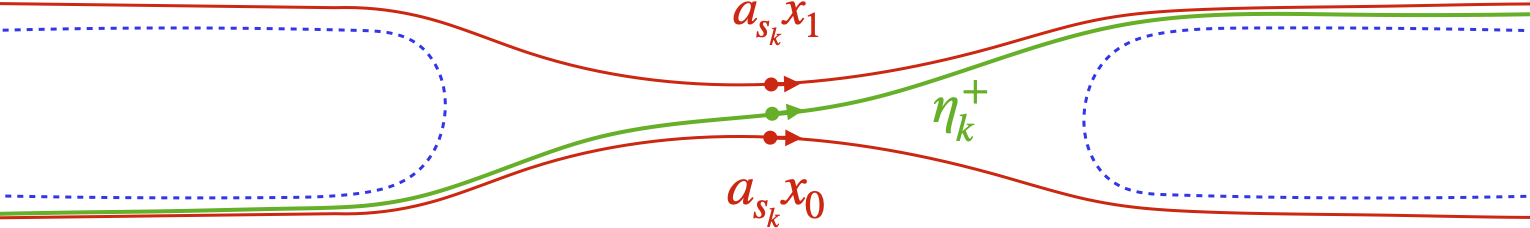}
		\vspace*{-0.5cm}
		\caption{\small }
		\label{fig: Crossings}
	\end{figure}
	
	Given a path $\alpha$, we denote by $\hat{\alpha}$ the geodesic segment given by pulling $\alpha $ tight, that is, the unique geodesic segment in $\alpha$'s homotopy class relative its endpoints (or endpoints at infinity). Given a weaving pattern $W$ we define
	\begin{equation}\label{eq:weaving prototype eta_W}
		\eta_{W,+}=\widehat{\eta_{k_1}^+\ast \eta_{k_2}^- \ast \eta_{k_3}^+ \ast...},
	\end{equation}
	where the $\ast $ notation is interpreted as concatenation of subsegments beginning and ending at the unique intersection points of the paths, see \Cref{fig: weaving geodesic}. We similarly define $\eta_{W,-}$ where the alternating signs of the crossing begin with a `$-$' sign. Whenever the weaving pattern is finite the concatenation ends with a full subray of the final crossing.
	\newpage
	
	\begin{figure}
		\centering
		\includegraphics[width=0.92\linewidth]{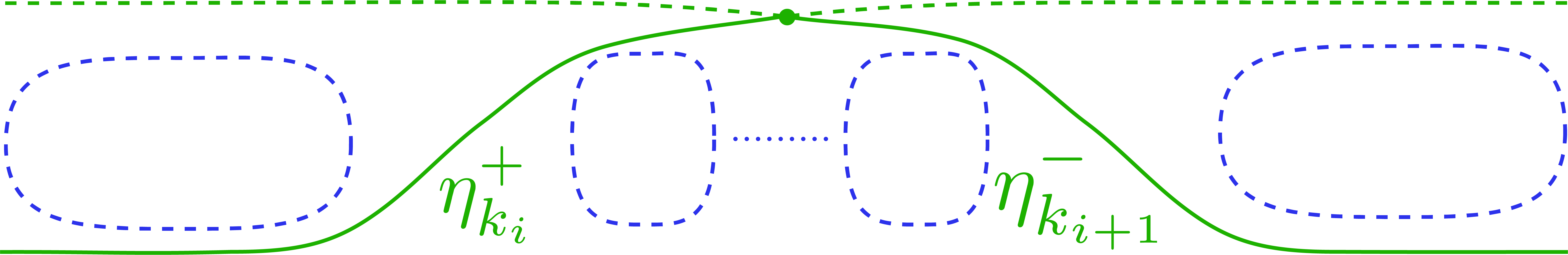}
		\caption{\small Intersection point of two crossings.}
		\label{fig: weaving geodesic}
	\end{figure}
	
	The weaving lemma allows us to draw the following realization lemma:
	\begin{lemma}\label{lem:weaving realization eta_W}
		Every weaving geodesic $A_+w$ with pattern $W$ is given by pulling tight $\alpha_0 \ast \eta_{W,\sigma}$ where $\alpha_0$ is any path connecting $p(w)$ to $\eta^{\sigma_1}_{k_1}$ fully contained within one $J_s\times\{j\}$ and $\sigma\in \{\pm\}$ is chosen appropriately.
	\end{lemma}
	
	\subsection{Slack of a Weaving Geodesic}
	
	As a first step we note that the slack of a single crossing is entirely determined by the ``height'' of the boundary component $h_k$. In fact using the hyperbolic cosine law one can verify the following:
	\begin{Fact}\label{Fact:equation slack to h_k}
		$\s(\eta_k^\pm) = 2\ln \cosh (d_{\Hplane}(0,i h_k)) = 2\ln \cosh \left(\int_0^{h_k}\frac{dt}{\cos(t)}\right) $.
	\end{Fact}
	
	We invoke the following lemma from \cite[Lemma 3.5]{farreClassificationHorocycleOrbit2024}:
	
	\begin{lemma}\label{lem:epsilon chain slack}
		For all $ c>0 $ there exist constants $\kappa_c,\varepsilon_0>0$ such that the following holds for all $ 0 < \varepsilon < \varepsilon_0 $. Let $ \alpha_i:[a_i,b_i] \to \Sigma$ for $ i=1,...,n $ be a sequence of geodesic arcs, each of length greater or equal to $c$, and satisfying
		\begin{equation}\label{eqn: sum jumps}
			\sum_{i=1}^{n-1} d_{\T^1\Sigma}(\T^1\alpha_i(b_i),\T^1\alpha_{i+1}(a_{i+1})) < \varepsilon
		\end{equation}
		and let $\bar\alpha$ denote an arc obtained from $\cup\alpha_i$ by joining each endpoint $\alpha_i(b_i)$ to $\alpha_{i+1}(a_{i+1})$ using arcs whose total length is less than $\varepsilon$. Then  there exists a geodesic arc $ \alpha $  homotopic, relative to the endpoints, to $\bar \alpha$ and satisfying
		\[ \left|\s(\alpha)- \sum_{i=1}^n \s(\alpha_i) \right| < \kappa_c \cdot \varepsilon. \]
		Moreover,  the Hausdorff distance between $ \alpha $ and $ \bar \alpha $ is smaller than $ \kappa_c \varepsilon $.
		
		The claim further holds with $ \alpha_n:[a_n,\infty) \to \Sigma
		$ and where $ \alpha $ is a geodesic ray from $ \alpha_1(a_1) $ 
		which is forward-asymptotic to $\alpha_n$; and similarly with
		$\alpha_1:(-\infty,b_1]\to\Sigma$. 
	\end{lemma}	
	
	The fact that the $h_k$'s are uniformly bounded above, together with the assumption that $d_{\Hplane}(\partial D_{h_k}(s_k), \partial D_{h_{k+1}}(s_{k+1})) \to \infty$ implies $|s_k-s_l| \to \infty$ if $k\not=\ell$ and $k,\ell \to \infty$ and that corresponding crossings $\eta_k^\pm$ and $\eta_l^\mp$ intersect at angles tending to 0. 
    Moreover, taking increasing and exhausting portions of the crossing will account for all of $\s(\eta_k^\pm)$, e.g.~$\s(\eta_k^\pm|_{[-T,T]})\to \s(\eta_k^\pm)$. We thus conclude the following:
	
	\begin{prop}\label{prop:finite weaving pattern large gaps slack is sum}
		For any $\varepsilon>0$ and $m \in \N$, there exists $S>0$ such that any length $m$ weaving pattern, $W=\{k_1 < k_2 <...<k_m\}$, satisfying 
		\[ |s_{k_j}-s_{k_{j+1}}|>S \quad\text{for all}\quad 1\leq j < m \]
		will have:
		\[ \left|\s(\eta_{W,\pm})-\sum_{j=1}^{m}\s\left(\eta_{k_j}^{\pm}\right)\right|<\varepsilon. \]
	\end{prop}
	
	A fundamental observation, highlighting the utility of the notion of slack, is the following lemma proven in \cite[Lemma 3.3]{farreClassificationHorocycleOrbit2024}:
	\begin{lemma}\label{lemma:recurrence iff ray with slack}
		For any $ x, y \in \T^1_+\lambda$ with $\tau(x) = \tau(y)$  and $ t \geq 0 $, we have $ a_t y \in \overline{Nx}$ if and only if there exists $ y_m \to y $
		such that $A_+y_m$ is asymptotic to $A_+x$, and $ \s(A_+y_m) \to t $. 
	\end{lemma}
	
	\noindent The elementary proof of this lemma applies verbatim to the setting of this paper. We will make explicit use of this lemma in \Cref{sec:distal}.
	
	\section{Minimal Orbit Closures}
	
	In this section we will work under the following assumption:
	
	\begin{definition}
		A loom surface $\Sigma_s$ with $s=(s_j,h_j)_{j \in \N}$ is said to satisfy the \emph{summability condition} if 
		\[ \sum_{j=1}^\infty \s(\eta_j^\pm)=\sum_{j=1}^\infty 2\ln \cosh (d_{\Hplane}(0,i h_j)) < \infty. \]
	\end{definition}
	
	We prove the following:
	\begin{theorem}\label{Thm:summable case}
		Let $\Sigma_s$ be a loom surface satisfying the summability condition, then the orbit closure $F=\overline{Nx_0}$ is $N$-minimal and satisfies
		\[ Nx_0 \subsetneq F \subsetneq \mathcal{E}, \]
		where $\mathcal{E}$ denotes the non-wandering set for the horocyclic flow on $\T^1\Sigma_s$.
	\end{theorem}
	\medskip
	
	Recall that two points $y,w \in \T^1\Sigma_s$ are called $A$-\emph{proximal} if  
	\[ \liminf_{t\to +\infty} d_{\T^1\Sigma_s}(a_ty,a_tw)=0. \]
	We will make use of the following fact, proven in \cite[Cor.~8.3]{farreMinimizingLaminationsRegular2023a}:
	\begin{prop}\label{prop:proximal points same OC}
		If $y,w \in \T^1 \Sigma_s$ are $A$-proximal then $\overline{Ny}=\overline{Nw}$.
	\end{prop}
	
	Our strategy for proving \Cref{Thm:summable case} will be to show that all the points in $\overline{Nx_0}$ are $A$-proximal to $x_0$. 
	\medskip
	
	Recall our notation for the projection $p:T^1\Sigma_s \to \Sigma_s$. The half planes 
	\[ \mathcal{O}=\left\{z \in \Hplane : \mathrm{Im}z<0 \right\}\times\left\{0,1\right\}, \]
	whose boundaries are $p(A_+x_0)$ and $p(A_+x_1)$ respectively, are isometrically embedded in $\Sigma_s$. Therefore if a geodesic ray $\ell_+ = p(A_+w)$ in $\Sigma_s$ enters $\mathcal{O}$ it can never leave it. We observe the following:
	\begin{lemma}[Bottleneck Lemma]\label{lemma:1st observation bottlenecks}
		There exists a sequence $\delta_k \to 0$ for which any $w\in \T^1\Sigma_s$ with $\sup \tau\left(A_+w\right)=+\infty$ and $p(A_+w) \subseteq \Sigma_s\smallsetminus \mathcal{O}$ satisfies
		\[ d_{\T^1\Sigma_s}\left(a_{t_k}w,a_{s_k}x_0\right) \leq \delta_k \quad \text{for all large }k, \]
		where $t_k \to \infty$ is the sequence of times in which $\tau(a_{t_k}w)=\tau(a_{s_k}x_0)=s_k$.
		
		More generally, if $A_+w$ is any geodesic ray with unbounded positive $\tau$-value then
		\begin{equation}\label{eq:non-uniform asymptoticity in bottlenecks}
			d_{\T^1\Sigma_s}\left(a_{t_k}w,a_{s_k}x_0\right) \to 0,
		\end{equation}
		for similarly chosen $t_k$'s.
	\end{lemma}
	
	\noindent This lemma describes a sequence of arbitrarily thin ``bottlenecks'', depicted in \Cref{fig: Crossings}, which every geodesic ray with unbounded positive $\tau$-value in $\Sigma_s\smallsetminus \mathcal{O}$ is forced to pass through. 
	
	\begin{proof}
		First, note that the summability condition implies in particular that $h_k \to 0$ and hence
		\begin{equation}\label{eq:proximality of x_0 and x_1}
			\lim_{k \to \infty} d_{\T^1\Sigma_s}\left(a_{s_k}x_0,a_{s_k}x_1\right)=0.
		\end{equation}		
		Let $I_k=[a_{s_k}x_,a_{s_k}x_1]$ be the ``vertical'' geodesic segment connecting the $a_{s_k}x_i$'s. Note that $I_k=\tau^{-1}(s_k)\cap \Sigma_s\smallsetminus \mathcal{O}$.
		
		Let $w\in \T^1\Sigma_s$ with $\ell_+=p(A_+w)\subseteq \Sigma_s\smallsetminus \mathcal{O}$ be as in the statement of the lemma. For all large $k$, the balls $B_k=B_3(a_{s_k}x_0)$ of radius 3 around $a_{s_k}x_0$ are isometrically embedded in $\Sigma_s$ and contain $\T^1I_k=p^{-1}(I_k)$. The geodesic segments $A_{[-1,1]}a_{s_k}x_i$ are contained in $B_k$ for all large $k$ and their Hausdorff distance tends to 0 in $\T^1\Sigma_s$. By our assumption, $\ell_+\cap B_k \subseteq B_k \smallsetminus \mathcal{O}$ implying that $\ell_+\cap I_k \neq \emptyset$, see \Cref{fig: Crossings}. One can readily verify that
		\[ \delta_k:=\textrm{diam}\left\{y \in \T^1I_k : p\left(A_{[-1,1]}y\right) \cap p\left(A_{[-1,1]}a_{s_k}x_0 \cup A_{[-1,1]}a_{s_k}x_1\right)=\emptyset\right\} \]
		tends to $0$ as $k\to \infty$. Since $p\left(A_+w\right) \cap p\left(Aa_{s_k}x_0 \cup Aa_{s_k}x_1\right)=\emptyset$ we have
		\begin{equation}\label{eq:subseq in bottlenecks}
			d_{\T^1\Sigma_s}\left(a_{t_k}w,a_{s_k}x_0\right) \leq \delta_k \qquad \text{for all large }k,
		\end{equation}
		where $a_{t_k}w \in \T^1I_k$, as claimed.
		
		More generally, notice that any geodesic ray $A_+w$ with unbounded $\tau$-value is either eventually contained in $\Sigma_s\smallsetminus \mathcal{O}$ or is asymptotic to one of the $A_+x_i$'s.	In both cases we conclude \eqref{eq:non-uniform asymptoticity in bottlenecks}.
	\end{proof}
	
	As a direct corollary we deduce:	
	\begin{cor}\label{Cor:From 1st observation beta determines proximality}
		If $\beta(w)=b>-\infty$ then $w$ is $A$-proximal to $a_b x_0$.
	\end{cor}
	
	\begin{proof}
		Assume $\beta(w)=b>-\infty$, then in particular $\tau(a_tw)\to +\infty$ as $t \to +\infty$ and hence \eqref{eq:non-uniform asymptoticity in bottlenecks} holds for $\tau(a_{t_k}w)=\tau(a_{s_k}x_0)$. By the definition of $\beta(w)$, in \eqref{eq:Definition of beta}, we have 
		\[ \lim_{k \to \infty}s_k-t_k=\lim_{k \to \infty}\tau(a_{t_k}w)-t_k=\beta(w)=b,  \]
		that is, $d_{\T^1\Sigma_s}\left(a_{t_k}w,a_{t_k+b}x_0\right) \to 0$, as claimed.
	\end{proof}
	
	A second useful observation is the following:
	\begin{lemma}\label{lemma:2nd observation - tail weaving pattern}
		For any $\varepsilon > 0$ there exists $k_0 \in \N$ for which all weaving patterns $W$ with $W\geq k_0$ satisfy $\s(\eta_{W,\pm})<\varepsilon$.
	\end{lemma}
	
	\begin{proof}
		By the summability condition, for any $\varepsilon>0$ there exists $k_0$ large enough for which 
		\[ \sum_{j=k_0}^\infty \s(\eta_j^\pm) < \varepsilon. \]
		Let $W=\{k_1,k_2,...\}$ be any weaving pattern with $W\geq k_0$. By the additivity and non-negativity of slack we conclude that the ``untightened'' broken path
		\[ \alpha_W=\eta_{k_1}^\pm \ast \eta_{k_2}^\mp \ast ..., \]
		given by concatenating the appropriate subsegments of each crossing has
		\[ \s(\alpha_W)\leq \sum_{j=k_0}^\infty \s(\eta_j^\pm)<\varepsilon. \]
		
		The weaving geodesic $\eta_{W,\pm}=\widehat{\alpha_W}$ is given by pulling $\alpha_W$ tight. By construction, both $\alpha_W$ and  $\eta_{W,\pm}$ are contained in $\Sigma_s\smallsetminus \mathcal{O}$ and have unbounded $\tau$-values, implying by \Cref{lemma:1st observation bottlenecks} the existence of $p^+_k \in \eta_{W,\pm}$ and $q^+_k \in \alpha_W$ satisfying
		\[ \tau(p^+_k)=\tau(q^+_k)=s_k \to +\infty \qquad \text{and}\qquad d_{\Sigma_s}(p^+_k,q^+_k) \to 0. \]
		Moreover, both paths are backwards asymptotic to one of the $A_-x_i$'s giving sequences $p^-_k$ and $q^-_k$ with equal $\tau$-values tending to $-\infty$ and having $d_{\Sigma_s}(p^-_k,q^-_k) \to 0$.
		
		Denote $\eta_{W,\pm}^k$ the subsegment of $\eta_{W,\pm}$ connecting $p^-_k$ to $p^+_k$, and by $\alpha_W^k$ the subsegment of $\alpha_W$ connecting $q^-_k$ to $q^+_k$. By construction, the length of $\eta_{W,\pm}^k$ is shorter than the length of $[p^-_k,q^-_k]\ast \alpha_W^k \ast [p^+_k,q^+_k]$ implying
		\begin{align*}
			\s(\eta_{W,\pm})&=\lim_{k \to \infty} \s(\eta_{W,\pm}^k)\leq\\
			&\leq \lim_{k \to \infty} \left[\s(\alpha_{W}^k)+d_{\Sigma_s}(p^-_k,q^-_k)+d_{\Sigma_s}(p^+_k,q^+_k)\right]=\\
			&=\s(\alpha_{W})+0+0<\varepsilon.
		\end{align*} 
	\end{proof}
	
	We are now set to prove the theorem:
	\begin{proof}[Proof of \Cref{Thm:summable case}]
		First, by $\eqref{eq:proximality of x_0 and x_1}$ we know that $x_0$ and $x_1$ are $A$-proximal, therefore implying $x_1 \in F=\overline{Nx_0}$ by \Cref{prop:proximal points same OC}. Since $x_1 \notin Nx_0$, as the corresponding $A_+x_j$'s are not asymptotic, we deduce that $F\neq Nx_0$. Recall that horocycles outside of $\mathcal{E}$, those who are based outside the limit set, are properly embedded in the quotient manifold, hence $Nx_0\subseteq \mathcal{E}$. Strict inclusion $F \subsetneq \mathcal{E}$ follows, for instance, from $x_0$ being quasi-minimizing, see \cite{eberleinHorocycleFlowsCertain1977,dalboTopologieFeuilletageFortement2000}.
		\medskip
		
		The main part of the claim is the $N$-minimality. As promised, our strategy will be to show that all accumulation points of $Nx_0$ are $A$-proximal to $x_0$ and therefore, by \Cref{prop:proximal points same OC}, have dense $N$-orbits in $F$. By \Cref{Cor:From 1st observation beta determines proximality}, it suffices to prove that $\beta(y)=0$ for all $y \in F$, which by definition amounts to showing:
		\begin{equation*}
			\s(A_+y)=\tau(y).
		\end{equation*}
		
		Let $y \in F$ and let $0<\varepsilon<1$ be smaller than half the minimal injectivity radius in $\Sigma_s$. By \Cref{lemma:weaving} and \Cref{lemma:2nd observation - tail weaving pattern} there exists $k_0$ large enough so that any geodesic ray $A_+w$ beginning at $\tau(w)\geq s_{k_0}$ with $\s(A_+w)\leq \rho=\tau(y)+1$ is weaving with pattern $W\geq k_0$ for which
		\begin{equation}\label{eq:proof of summable case - small slack tail}
			\s(\eta_{W,\pm})<\varepsilon.
		\end{equation}
		By increasing $k_0$ if necessary and using \Cref{lemma:1st observation bottlenecks} we may ensure $\delta_k<\varepsilon$ for all $k\geq k_0$. 
		
		Now fix $T\geq 0$ for which $d_{\T^1\Sigma_s}(a_Ty,a_{s_k}x_0)<\varepsilon$ with $\tau(a_Ty)=s_k>s_{k_0}+1$ for some $k> k_0$ as ensured by \Cref{lemma:1st observation bottlenecks}. By increasing $T$ if necessary we may further require that $\s(A_{[T,\infty)}y)<\varepsilon$, as this expression decreases to 0 in $T$.
		
		Let $nx_0 \in Nx_0$ be so close to $y$ so as to satisfy
		\begin{equation}\label{eq:summable case proof - shadow initial part}
			d_{\T^1\Sigma_s}(nx_0,y)<\varepsilon \quad\text{and}\quad d_{\T^1\Sigma_s}(a_Tnx_0,a_Ty)<\varepsilon.
		\end{equation}
		Denote $w=a_Tnx_0$.	The function $\beta$ being $N$-invariant implies $\beta(nx_0)=\beta(x_0)=0$ and hence 
		\begin{equation}\label{eq:summable case proof - slack=tau nx_0}
			\s(A_+nx_0)=\tau(nx_0).
		\end{equation}
		We thus have
		\begin{align*}
			\s(A_+w)&\leq \s(A_+nx_0)=\tau(nx_0)\leq \\
			&\leq \tau(y)+d_{\T^1\Sigma_s}(nx_0,y)\leq \tau(y)+1=\rho,
		\end{align*}
		by the 1-Lipschitz property of $\tau$. Since $\tau(w)>s_k-1>s_{k_0}$ we are ensured by \Cref{lemma:2nd observation - tail weaving pattern} that $A_+w$ is weaving of pattern $W \geq k_0$  satisfying \eqref{eq:proof of summable case - small slack tail}. Let $\eta_{W,\sigma}$ correspond to $A_+w$ as in \Cref{lem:weaving realization eta_W}.
		By \Cref{lemma:1st observation bottlenecks} and our choice of $a_Ty$ we know that
		\[ d_{\Sigma_s}(p(w),\eta_{W,\sigma})\leq d_{\Sigma_s}(p(w),p(a_Ty))+d_{\Sigma_s}(p(a_Ty),\eta_{W,\sigma})<\varepsilon+\delta_k<2\varepsilon.  \]
		Let $\alpha$ denote the shortest path connecting $p(w)$ to $\eta_{W,\sigma}$, then $\mathrm{length}(\alpha)<2\varepsilon$ is smaller than the injectivity radius around $p(w)$ implying that $A_+w$ is given by pulling tight $\alpha \ast \eta_{W,\sigma}$. We thus conclude
		\begin{align*}
			\s(A_+w)&\leq 2\mathrm{length}(\alpha)+\s(\eta_{W,\sigma})<4\varepsilon+\varepsilon<5\varepsilon,
		\end{align*}
		where we made use of the fact that the slack of a path is always bounded above by twice its length, since $\tau$ is $1$-Lipschitz.		
		
		On the other hand, by \eqref{eq:summable case proof - shadow initial part} we have
		\[ |\s(A_{[0,T]}y)-\s(A_{[0,T]}nx_0)|\leq |\tau(y)-\tau(nx_0)|+|\tau(a_Ty)-\tau(a_Tnx_0)|<2\varepsilon. \]
		We therefore obtain
		\begin{align*}
			 |\s(A_+y)-\s(A_+nx_0)|&\leq 2\varepsilon+\s(A_{[T,\infty)}y)+\s(A_{[T,\infty)}nx_0)=\\
			 &=2\varepsilon+\varepsilon+\s(A_+w)<8\varepsilon.
		\end{align*}
		By \eqref{eq:summable case proof - slack=tau nx_0} we thus have
		\[ |\s(A_+y)-\tau(y)|\leq |\s(A_+y)-\s(A_+nx_0)|+|\tau(nx_0)-\tau(y)|<9\varepsilon. \]
		Since $\varepsilon$ was arbitrary we conclude $\s(A_+y)=\tau(y)$, as claimed.
	\end{proof}
	
	\begin{remark}\label{rmk:minimal set on beta level set}
		We have shown that $\beta$ is constant (zero) on all of $\overline{Nx_0}$. This also follows from the minimality of $\overline{Nx_0}$ together with \eqref{eq:beta horoball}.
		Moreover, the fact that $\beta(a_tz)=\beta(z)+t$ for all $t \in \R$ implies that $a_t y \notin \overline{Nx_0}$ for all $y \in \overline{Nx_0}$ and $t \neq 0$.
	\end{remark}
		
	\subsection{Invariant Measures} In this subsection we show the following:
	\begin{thm}\label{Thm:Existence of new horo measure}
		Let $\Sigma_s$ be a loom surface satisfying the summability condition, then the minimal orbit closure $\overline{Nx_0}$ supports a locally finite $N$-invariant and ergodic measure, $\mu$, which is infinite, conservative and singular with respect to the geodesic flow, that is, 
		\[ a_t.\mu \perp \mu \quad\text{for all } t\neq 0. \]
	\end{thm}
	
	This theorem provides the first examples of $N$-ergodic and invariant locally finite measures which are neither quasi-invariant with respect to the geodesic flow nor supported on a single closed horocycle.
	\medskip
	
	\begin{proof}[Proof of \Cref{Thm:Existence of new horo measure}]
		The existence of an $N$-invariant locally finite measure supported on  $\overline{Nx_0}$ follows from \cite[Lemma 2.2]{kellerhalsNonsupramenableGroupsActing2013} and the superamenability of the group $\R$ \cite{rosenblattInvariantMeasuresGrowth1974}. See also \cite{MR4498410}. For completeness and accessibility, we provide a short, more hands-on, construction of such measures in \Cref{prop: appendix measure construction} in the appendix.
		
		These statements ensure the existence of an  $N$-invariant, conservative and locally-finite measure supported on $\overline{Nx_0}$. Let $\mu$ denote a locally-finite ergodic component of said measure. By Ratner's classification of finite $N$-invariant measures \cite{ratnerRaghunathansConjecturesRm1992} (and the fact that $\overline{Nx_0}$ is not a homogeneous subspace of $\T^1\Sigma_s$) we conclude $\mu$ is necessarily infinite.
		
		Singularity of $\mu$ with respect to the geodesic flow follows from the fact that $a_t\overline{Nx_0} \cap \overline{Nx_0} = \emptyset$ for all $t \neq 0$, see \Cref{rmk:minimal set on beta level set}, implying that the topological support of $\mu$ is not invariant under any $a_t$ with $t \neq 0$.
	\end{proof}
		
	\begin{remark}\label{rem:Patterson example}
		We note that the notion of a ``trivial'' measure in this paper is strictly broader than the one used by Sarig in \cite{sarigHorocyclicFlowLaplacian2010}, who referred only to measures corresponding to periodic horocycles or orbits based outside the limit set. The measures constructed in \cite[Theorem 3]{sarigHorocyclicFlowLaplacian2010} are trivial in the sense that they are supported on single proper horocycle orbits, which was known by the author. 
		
		The measures constructed in \cite[Theorem 3]{sarigHorocyclicFlowLaplacian2010} are ergodic components of Lebesgue measure on certain surfaces given by Fuchsian groups of the first kind with critical exponent$<\!1/2$, see \cite{pattersonExamplesFuchsianGroups1979}.		
		In \cite[Equation (18)]{pattersonSpectralTheoryFuchsian1977}, Patterson shows that Fuchsian groups $\Gamma$ having critical exponent smaller than one-half satisfy, in the Poincar\'e disk model, 
		\[ \sum_{\gamma \in \Gamma} |\gamma'(\xi)| < \infty \quad \text{for Lebesgue-almost every }\xi \in S^1. \]
		Recall that in the disk model one has $|\gamma'(\xi)|=e^{-\beta_{\xi}(\gamma^{-1}.0,0)}$ where $\beta$ is the Busemann cocycle, see e.g.~\cite[Appendix 1]{sarigHorocycleFlowsSurfaces2019}. Therefore, we have that for Leb-a.e.~$\xi \in \partial \Hplane$, any horocycle tangent to $\xi$ has only finitely many points in $\Gamma.0$ within a bounded distance of it. This in turn implies that such horocycles have no accumulation points in the quotient surface.  
	\end{remark}
	
	\section{Fractional Hausdorff Dimension}\label{sec:distal}
	
	In this section we consider surfaces satisfying the following:
	\begin{definition}
		A loom surface $\Sigma_s$ with $s=(s_j,h_j)_{j \in \N}$ is called \emph{distal} if 
		\[ \inf_{j \in \N} h_j > 0. \]
	\end{definition}
	\noindent Note that under these conditions, the two points $x_0$ and $x_1$ are $A$-distal, i.e.~they satisfy $\inf_{t \to \infty} d(a_tx_0,a_t x_1)>0$.
	\medskip
	
	We observe that \Cref{Fact:asymp_to_lambda} implies, in this case, that any finite slack geodesic ray is eventually asymptotic to either $A_+x_0$ or $A_+x_1$ since the $\varepsilon$-neighborhood of $\T^1_+\lambda$ is disconnected for small enough $\varepsilon$. In other words, all points in $\T^1\Sigma_s$ supporting finite slack rays are contained in one of two distinct $AN$-orbits $ANx_0\sqcup ANx_1$. The function $\beta$ is $N$-invariant and also, by the definition, $A$-equivariant in the sense that for all $y$ and $t \in \R$ we have $\beta(a_ty)=\beta(y)+t$. \Cref{Fact:beta horoball supset} hence implies:
	\begin{lemma}\label{lemma:orbit closure decomp to two P+ orbits}
		If $\Sigma_s$ is a distal loom surface, then
		\[ \overline{Nx_0} \subseteq A_+Nx_0\sqcup A_+Nx_1. \]
	\end{lemma}
	
	In this section we are interested in the Hausdorff dimension of $\overline{Nx_0}$, we may thus consider each component of the orbit closure separately. Each $A_+N$-orbit in $\T^1\Sigma_s$ is a locally isometric projection of countably many $A_+N$-orbits in $\T^1\Hplane$. Each lift of $\overline{Nx_0} \cap A_+Nx_j$ is a product set which is $N$-saturated, implying
	\[ \dim_\mathrm{H}\overline{Nx_0}\cap A_+Nx_j = 1+\dim_\mathrm{H}\overline{Nx_0}\cap A_+x_j. \]
	Hence our analysis boils down to understanding the sets
	\[ \Delta_j=\{t \geq 0 : a_tx_j \in \overline{Nx_0}\cap A_+x_j\}. \]
	See \cite[\S2.2]{farreClassificationHorocycleOrbit2024} and \cite[\S7]{farreMinimizingLaminationsRegular2023a} for another account of such sets (under different notations).
	\medskip
	
	The main theorem of this section, proven in \S\ref{Subsec:proof of main distal thm}, is the following:
	\begin{theorem}\label{thm:main distal surface}
		Let $\Sigma_s$ be a distal loom surface. Let $E=\mathrm{accum}\left(\s(\eta_k^\pm)\right)_{k \in \N}$ be the accumulation points in $\R$ of the sequence of slacks of crossings. Then
		\begin{equation}\label{eq:distal thm Delta_j=sums}
			\Delta_0=\bigcup_{k =1}^\infty 2kE \qquad \text{and}\qquad \Delta_1=\bigcup_{k=0}^\infty (2k+1)E,
		\end{equation}
		where $mE:=\underbrace{E+\cdots+E}_m$.
	\end{theorem}
	
	\begin{remark}
		Note that the distality assumption ensures $E>\delta > 0$ for some $\delta$ and hence for all $T>0$
		\[ \bigcup_{m \in \N} mE \cap [0,T] \subseteq \bigcup_{m=1}^{\lceil \frac{T}{\delta}\rceil} mE, \]
		implying in particular that both $\Delta_j$ are closed.
	\end{remark}
	
	The vast flexibility of our construction allows us to exhibit any compact\footnote{Our definition of loom surfaces is not the most general imaginable. In particular, the assumption that the $h_k$ are uniformly bounded away from $\frac{\pi}{2}$ is merely imposed for simplicity. One can construct surfaces with an unbounded sequence of $h_k$ under which additional conditions on $|s_{k+1}-s_k|\gg 0$ would ensure all claims in this paper will hold.} subset of $(0,\infty)$ as $E$, that is, we have:
	\begin{prop}\label{prop: distal loom surface slacks}
		For any compact set $E \subset (0,\infty)$ there exists a distal loom surface $\Sigma_s$ in which $E=\mathrm{accum}\left(\s(\eta_k^\pm)\right)_{k \in \N}$.
	\end{prop}
	
	\begin{proof}
		Given $E$, choose a finite or countable dense subset $\{e_1,e_2,...\}$ of $E$. By solving for $h$ in \Cref{Fact:equation slack to h_k} we can construct $\Sigma_s$ where every $e_j$ appears infinitely many times as the slack of different crossings. The proposition follows.
	\end{proof}
	
	\subsection{Examples}
	These results provide a rich source of interesting and non-regular examples. Here are a few:
	\subsubsection{Fixed Fractional Hausdorff Dimension}
	For any $\alpha \in [0,1]$ there exists a compact set $E'$ with $\dim_{\mathrm{H}}mE'=\alpha$ for all $m \geq 1$, see \cite{schmelingDimensionIteratedSumsets2010,kornerHausdorffDimensionSums2008}. 
    Since $E'$ may contain $0$, and the accumulations of slacks of geodesics $\eta_k^\pm$ are all strictly positive in a \emph{distal} loom surface, we take $E=1+E'$.  By Proposition \ref{prop: distal loom surface slacks}, there is a distal loom surface with $E=\mathrm{accum}\left(\s(\eta_k^\pm)\right)_{k \in \N}$.  Since $\dim_\mathrm{H} mE = \dim_\mathrm{H} mE'$, \Cref{thm:main distal surface} gives the following:
	\begin{cor}
		For any $\alpha \in [0,1]$ there exists a distal loom surface having $$\dim_\mathrm{H}\overline{Nx_0}=1+\alpha.$$
	\end{cor}
	
	\subsubsection{Locally Varying Dimensions}\label{subsec:varying dimensions}

    In fact, in \cite{schmelingDimensionIteratedSumsets2010} the authors construct compact sets $E'\subset [0,1]$ satisfying:
	\[ \alpha_m=\dim_\mathrm{H}mE' \quad,\quad \beta_m=\underline{\dim}_\mathrm{M}mE' \quad,\text{ and}\quad \gamma_m=\overline{\dim}_\mathrm{M}mE', \]
	where $\overline{\dim}_\mathrm{M}$ and $\underline{\dim}_\mathrm{M}$ denote the upper and lower Minkowski dimensions, respectively, and where 
	\[ 0\leq \alpha_m\leq \beta_m \leq \gamma_m \leq 1 \]
	are \emph{arbitrary} non-decreasing sequences (where $\{\beta_m\}$ and $\{\gamma_m\}$ are required to satisfy some mild growth constraints).  See also \cite{kornerHausdorffDimensionSums2008}.

    Applying \Cref{prop: distal loom surface slacks} and \Cref{thm:main distal surface} to $E = E'+1$, we can produce loom surfaces supporting horocyclic orbit closures having locally varying Hausdorff dimensions and such that the Hausdorff dimension disagrees with the Minkowski dimensions locally, i.e., the sets $\overline{Nx_0} \cap B_r(a_tx_0)$ have varying different dimensions depending on $t \geq 0$ and $r>0$.\footnote{It is not immediately evident what the local Minkowski dimensions are, as these do not comport well with countable unions. Nevertheless, choosing $\alpha_m < \beta_m$ implies strict inequalities for $\dim_{\mathrm{H}} < \overline{\dim}_{\mathrm{M}}$.}
	
	\subsubsection{Discrete Sub-Invariance}
	Elements of $\Delta_0$ correspond to sub-invariance symmetry of $\overline{Nx_0}$, that is, $t \in \Delta_0$ are exactly those numbers for which
	\[ a_t \overline{Nx_0} \subset \overline{Nx_0}, \]
	see \cite[\S7]{farreMinimizingLaminationsRegular2023a}. By constructing $\Sigma_s$ with $E=\{T_0\}$ for some $T_0>0$ we obtain a horocyclic orbit closure with a discrete semigroup of sub-invariance. This is in contrast with orbit closures in $\Z$-covers of compact surfaces, see \cite[Prop. 7.20]{farreMinimizingLaminationsRegular2023a}.
	
	\subsection{Proof of \Cref{thm:main distal surface}}\label{Subsec:proof of main distal thm}
	We proceed with the proof of the main theorem of this section. A key component is \Cref{lemma:recurrence iff ray with slack}, which can rephrased in our setting as follows:
	\begin{lemma}\label{lemma:recurrence iff ray with slack_loom surface}
			For $j=0,1$ and $t \geq 0$, the point $a_tx_j$ is contained in $\overline{Nx_0}$ if and only if there exists $y_m \to x_j$ with $A_+y_m$ asymptotic to $A_+x_0$ and  $\s(A_+y_m)\to t$.
	\end{lemma}
	
	\begin{proof}[Proof of \Cref{thm:main distal surface}]
		Let us first consider the claim
		\[ \Delta_0=\bigcup_{k \in \N} 2kE. \]
		Let $t \in 2kE$ with $t=\sum_{j=1}^{2k} e_j$ where $e_j \in E$. Let $\varepsilon>0$ and let $S>0$ be the constant in \Cref{prop:finite weaving pattern large gaps slack is sum} corresponding to $\varepsilon$ and $m=2k$. Let $(s_k)$ be the sequence of horizontal positions of boundary components in the definition of $\Sigma_s$. Let $T>0$ be large enough that $|s_k-s_l|>S$ holds for all distinct $s_k,s_l > T$.
		By the definition of $E$ there exists for each $e_j$ a crossing $\eta^\pm_{k_j}$ with $|\s(\eta_{k_j}^\pm)-e_j|<\varepsilon/2k$ and with $s_{k_j}>T$. Hence the weaving geodesic $\eta_W^+$ with weaving pattern $\{k_1,...k_{2k}\}$ satisfies the conditions of \Cref{prop:finite weaving pattern large gaps slack is sum} implying:
		\[ \left|\s(\eta_W^+)-\sum_{j=1}^{2k}e_j\right| < \left|\s(\eta_W^\pm)-\sum_{j=1}^{2k}\s(\eta^\pm_{k_j})\right|+\varepsilon <2\varepsilon. \]
		Notice that the weaving geodesic $ \eta_W^+$ is both backward and forward asymptotic to $Ax_0$. In particular, as $T$ increases the distance between $x_0$ and such $\eta_W^+$ tends to 0. We may thus choose $T$ so large such that $d(x_0,y)<\varepsilon$ for some $y\in \eta_W^+$ satisfying $|\s(A_+y)-\s(\eta_W^+)|<\varepsilon$. Since $\varepsilon>0$ was arbitrary we conclude from \Cref{lemma:recurrence iff ray with slack_loom surface} that $a_tx_0 \in \overline{Nx_0}$, i.e.~that $t\in\Delta_0$.
		
		In the other direction, let $t \in \Delta_0$ then by \Cref{lemma:recurrence iff ray with slack_loom surface} there exist a sequence of $y_m \to x_0$ satisfying $\s(A_+y_m) \to t$. Denote $\rho=t+1$ and let $S>0$ be the constant from \Cref{lemma:weaving}. For all $0<\varepsilon<1$ and all large enough $m$ we have $\s(A_+y_m)<\rho$ and $d(a_{S+1}y_m,a_{S+1}x_0)<\varepsilon$ implying in particular that $\tau(a_{S+1}y_m)>S$. Hence by \Cref{lemma:weaving} we conclude that $A_{[S+1,\infty)}y_m$ and hence $A_+y_m$ are weaving geodesic rays. Since $\Sigma_s$ is distal we conclude that $A_+y_m$ has a finite weaving pattern. Moreover, after a possible arbitrarily small perturbation (along a short expanding horocycle) we may assume that $y_m$ is backward asymptotic to $Ax_0$. Therefore $Ay_m=\eta_W^+$ for some finite even weaving pattern $W$. By identical considerations as before we can approximate $t\approx \s(A_+y_m)\approx \s(\eta_W^+)$ by an even sum of slacks of crossings, as claimed.

        Similar considerations yield the description of $\Delta_1$ in the statement of the theorem.  
	\end{proof}
	
	\newpage
    \appendix
	\section{Construction of Invariant Measure}\label{Appendix}
	
	In this section we prove the existence of a locally finite conservative $N$-invariant measure supported on the $N$-minimal set constructed in \Cref{Thm:summable case}. We first prove a more general proposition concerning continuous flows having a ``nice'' section, and then move on to apply the proposition to our setting.
	\medskip
	
	Let $X$ be a locally compact second countable Hausdorff space together with a continuous $\R$-action (flow) $\varphi_t$.
    Given a subset $I$ of $\R$ and $F \subseteq X$ we denote $\varphi_I F := \{ \varphi_t x: t \in I, x \in F \}$.
    A subset $Y \subseteq X$ is called $\varphi$-minimal if $Y=\overline{\varphi_{\R} y}$ for all $y \in Y$. 
    A (partial) section $\Psi\subseteq X$ for the $\varphi$-flow is a subset satisfying that $\left\{ t \in \R : \varphi_t(x) \in \Psi \right\}$ is either discrete in $\R$ or empty, for all $x \in X$.
    
    \begin{prop}\label{prop: appendix measure construction}
    	Let $Y\subseteq X$ be a $\varphi$-minimal set consisting of more than one $\varphi$-orbit. Assume there exists a precompact section $\Psi \subseteq X$ satisfying the following:
    	\begin{enumerate}
    		\item for all $R>0$ the set $\varphi_{(-R,R)}\Psi$ is open in $X$; and 
    		\item $Y\cap \overline{\Psi}=Y\cap \Psi \neq \emptyset$.
    	\end{enumerate}
    	Then $Y$ supports a locally finite $\varphi$-invariant conservative measure.
    \end{prop}
    
    Whenever the set $Y$ is compact then the classical Krylov–Bogolyubov theorem ensures the existence of an invariant probability measure. The proposition above thus deals  with the case where $Y$ is non-compact and where the time orbits spend in any compact set may be of zero density. 
    
    \begin{proof}
    	As explained above, we may assume that $Y$ is non-compact.
    	Fix some $y_0 \in Y$ and denote the open set $B_R=\varphi_{(-R,R)}\Psi$ for any $R>0$. We begin by observing that the orbit $\varphi_\R y_0$ spends an infinite amount of time in $B_R$ for any $R>0$. Indeed, since $B_R$ is open and intersects $Y$ non-trivially we know by the minimality of $Y$ and the fact that $Y$ is not one orbit, that $\varphi_\R y_0$ returns in an unbounded set of times to $B_R$. Each passage of $\varphi_\R y_0$ in $B_R$ is at least of length $2R$, implying 
    	\[ \left|\{t \in \R : \varphi_t y_0 \in B_R\}\right| = \infty. \]
    	    		
    	Given $R>0$, consider the following ``statistical'' probability measures:
    	\[ \nu^R_T = \frac{\int_{-T}^T \mathbbm{1}_{B_R}(\varphi_t y_0) \delta_{\varphi_t y_0}dt}{\int_{-T}^T \mathbbm{1}_{B_R}(\varphi_t y_0)dt} \]
    	where, $\delta_x$ denotes Dirac measure at $x \in X$. In other words, $\nu^R_T$ measures normalized arc-length along $\varphi_{(-T,T)}y_0 \cap B_R$.
    	\medskip 
    	
   		Fix $R>0$. We claim the following:
   		\medskip
   		
   		\noindent \underline{\textbf{Claim 1:}} The family of measures $\nu^R_T$ is asymptotically tight in $B_R$, that is, for every $\varepsilon>0$ there exists a compact subset $K_\varepsilon \subset B_R$ satisfying
   		\[ \nu^R_T(K_\varepsilon)\geq 1-\varepsilon \quad \text{for all large enough } T>0. \] 
    	\medskip 
    	
    	Given $0<\varepsilon$ set $0< \eta < \frac{1}{4}\varepsilon \cdot R$ and 
    	\[ K_\varepsilon =Y \cap \overline{B_{R-\eta}}. \]
    	By construction $K_\varepsilon$ is compact. Assumption (2) of the proposition implies
    	\[ K_\varepsilon =Y \cap \varphi_{[-R+\eta,R-\eta]}\overline{\Psi} = Y \cap \varphi_{[-R+\eta,R-\eta]}\Psi, \]
    	that is, $K_\varepsilon$ is contained in $B_R$.

    	For any $T>0$, the set of $t$'s in $(-T,T)$ for which $\varphi_t y_0 \in B_R$ is a union of $m$ open sub-intervals $I_1,...,I_m$ (finitely many), each of which is either bounded by $\pm T$ or is of length$\geq 2R$. Notice that by the definition of $B_R$, whenever $(t-2\eta,t+2\eta) \subseteq I_j$ for some $t$ then $\varphi_t y_0=\varphi_s z$ for some $s \in (-R+\eta,R-\eta)$ and $z \in \Psi$, i.e.~$\varphi_t y_0 \in B_{R-\eta}$. Hence, for each $1 \leq j \leq m$ 
    	\[ \left|\left\{t \in I_j : \varphi_t y_0 \in B_R \smallsetminus B_{R-\eta}\right\}\right| \leq 4\eta. \]
    	We therefore have
    	\begin{align*}
    		1-\nu^R_T(K_\varepsilon) &\leq \nu^R_T (B_R \smallsetminus B_{R-\eta}) \leq \\ 
    		&\leq \frac{1}{\sum_j |I_j|} \sum_j \left|\left\{t \in I_j : \varphi_t y_0 \in B_R \smallsetminus B_{R-\eta}\right\}\right| \leq \\
    		&\leq \frac{4\eta \cdot m}{2R \cdot (m-2)} < \frac{\varepsilon}{2} \cdot \frac{m}{m-2},
    	\end{align*}
    	where $m-2$ comes from omitting up to two $I_j$'s bounded by $\pm T$. Since $Y$ was assumed to be non-compact and $B_R$ is precompact, we know that $\varphi_\R y_0$ exits $B_R$ infinitely many times. Hence as $T \to \infty$ we have $m \to \infty$, implying
    	\[ \nu^R_T(K_\varepsilon) \geq 1-\varepsilon \]
    	for all large enough $T>0$, proving claim 1.
    	\medskip
    	
    	Asymptotic tightness implies that for any $R>0$ and sequence $T_n \to \infty$ there exists a subsequence for which the measures $\nu^R_{T_n} $ converge to a probability measure supported on $Y \cap B_R$.
    	\medskip
    	
    	\noindent \underline{\textbf{Claim 2:}} 
    	Any such limiting measure $\nu$ is invariant along flow lines in $B_R$.
        That is, if $ E \subseteq B_R $ is a Borel set, $s \in \R$, and $\varphi_s E \subseteq B_R$, then
        \[  \nu(E)=\nu(\varphi_s E). \]
    	    	
    	This follows from a standard amenability argument. It suffices to show that for any $f \in C_c(B_R)$ with $f \circ \varphi_s \in C_c(B_R)$ then $\nu(f \circ \varphi)=\nu(f)$. If $\mathrm{supp}(f)\cap Y = \emptyset$ the claim is trivially true. Otherwise, for any $T>0$ we have
    	\[ |\nu^R_T(f)-\nu^R_T(f\circ \varphi_s)|\leq \frac{2s \cdot \|f\|_\infty}{\int_{-T}^T \mathbbm{1}_{B_R}(\varphi_t y_0)dt}, \]
    	and since the denominator tends to infinity as $T \to \infty$ the claim follows.
    	\medskip
    	
    	We are ready to begin the construction in earnest. Let $T_n \to \infty$ be a sequence for which $\nu^1_{T_n} $ converges to a probability measure $\mu_1$ supported on $Y \cap B_1$. Let $T_{n_k}$ be a subsequence for which $\nu^2_{T_{n_k}} $ converges to a probability measure $\tilde{\mu}_2$ on $Y \cap B_2$. Taking a further subsequence we get a measure $\tilde{\mu}_3$ on $Y \cap B_3$ and so on ad infinitum. By a diagonal argument we obtain a sequence $T'_n$ under which 
    	\[ \nu^1_{T'_n} \to \mu_1 \quad \text{and} \quad \nu^\ell_{T'_n} \to \tilde{\mu}_\ell \quad \text{for all }2 \leq \ell \in \N. \]
    	
    	Since each $\tilde{\mu}_\ell$ is flow-line invariant inside $B_\ell$ and since finitely many $\varphi$-translates of $B_1$ cover $B_\ell$ we conclude that $\tilde{\mu}_\ell(B_1)>0$ for all $\ell$. Define
    	\[ \mu_\ell = \frac{1}{\tilde{\mu}_\ell (B_1)}\tilde{\mu}_\ell. \]
    	    	
    	\noindent \underline{\textbf{Claim 3:}} 
    	For any $1 \leq \ell_1<\ell_2$
    	\[ \mu_{\ell_2}|_{B_{\ell_1}}=\mu_{\ell_1}. \] 
    	\medskip
    	
    	Note that the inclusion of open sets $B_{\ell_1} \subseteq B_{\ell_2}$ gives a natural inclusion of $C_c(B_{\ell_1}) $ in $C_c(B_{\ell_2})$. By the definition of the statistical measures we thus have for any $T>0$ a constant $C(T)>0$ satisfying
    	\begin{equation}\label{eq:appdx restricted measures}
    		\nu^{\ell_2}_T(f)=C(T)\cdot \nu^{\ell_1}_T(f)
    	\end{equation}
    	for all $f \in C_c(B_{\ell_1})$. In fact
    	\[ C(T)= \frac{\int_{-T}^{T}\mathbbm{1}_{B_{\ell_1}}(\varphi_t y_0)dt}{\int_{-T}^{T}\mathbbm{1}_{B_{\ell_2}}(\varphi_t y_0)dt}.\]
    	Since the measures on both sides of \eqref{eq:appdx restricted measures} converge along $T'_n \to \infty$ we conclude that $C(T'_n)\to C$, where $C$ is independent of $f$, and 
    	\[ \tilde{\mu}_{\ell_2}(f)=C \cdot \tilde{\mu}_{\ell_1}(f) \qquad \text{for all } f \in C_c(B_{\ell_1}),\]
    	or
    	\[ \tilde{\mu}_{\ell_2}|_{B_{\ell_1}}=C \cdot \tilde{\mu}_{\ell_1}. \]
    	But after normalization both $\mu_{\ell_1}$ and $\mu_{\ell_2}$ give $B_1 $ (a subset of $ B_{\ell_1}$) mass 1, implying the claim.
    	\medskip    	
    	
    	We are ready to define the $\varphi$-invariant measure $\mu$ as follows:
    	\[ \mu(E):=\lim_{\ell \to \infty} \mu_\ell(E \cap Y) \]
    	for any Borel set $E \subseteq X$. Claim 3 implies that the above is an increasing limit, therefore ensuring this function is well-defined. Claim 3 further implies $\mu$ is indeed $\sigma$-additive and that
    	\begin{equation}\label{eq:appdx restriction final mu}
    		\mu|_{B_\ell} = \mu_\ell \quad \text{for all } \ell.
    	\end{equation}
    	
    	Crucially, the minimality of $Y$ together with assumptions (1)+(2) imply that $\left\{B_\ell\right\}_\ell$ is an open cover of $Y$. Hence, given any compact set $K \subset X$ there exists some $\ell$ large enough so that the compact set $K \cap Y $ is contained in $ B_\ell$. As $\mu(B_\ell)<\infty$, by construction, we thus conclude that $\mu$ is locally finite (and hence Radon, as $X$ is locally compact second countable and therefore also $\sigma$-compact).
    	
    	Additionally, given any compact set $K$ and any $s \in \R$, there exists $\ell$ so large such that $(K \cup \varphi_s K)\cap Y \subset B_\ell$. Hence by claim 2 and \eqref{eq:appdx restriction final mu}, and by the inner regularity of $\mu$, we conclude that $\mu$ is $\varphi$-invariant.
    	
    	Conservativity of $\mu$ follows from the fact that any locally finite dissipative measure is necessarily supported on properly embedded $\varphi$-orbits. Our assumption that $Y$ contains more than one orbit excludes this possibility.
    \end{proof}
    
    \subsection*{Existence of appropriate section}
    We now focus on the particular setting of this paper. Let $\Sigma_s$ be a loom surface satisfying the summability condition and let $Y=\overline{Nx_0}$ be $N$-minimal as discussed above.
    
    Let $U\leq \PSL_2(\R)$ denote the upper unipotent subgroup generating the expanding horocycle flow on $\T^1\Sigma_s$, we accordingly have
    \[ N=\left\{n_s = \begin{pmatrix}
    	1 & 0 \\ s & 1
    \end{pmatrix} : s \in \R \right\} \qquad,\qquad U=\left\{u_r = \begin{pmatrix}
    	1 & r \\ 0  & 1
    \end{pmatrix} : r \in \R \right\}. \]
    Given a subset $J \subseteq \R$, we denote $N_J=\{n_s : s \in J\}$, $U_J=\{u_r : r \in J\}$, and $A_J=\{a_t : t \in J\}$. 
    
    \begin{lemma}
    	There exists a section $\Psi \subseteq \T^1\Sigma_s$ for the $N$-flow satisfying:
    	\begin{enumerate}
    		\item for all $R>0$ the set $N_{(-R,R)}\Psi$ is open in $\T^1\Sigma_s$; and 
    		\item $Y\cap \overline{\Psi}=Y\cap \Psi \neq \emptyset$.
    	\end{enumerate}
    \end{lemma}
    
    This lemma together with \Cref{prop: appendix measure construction} concludes the proof of the existence of an $N$-invariant locally finite conservative measure, as desired.
    
    \begin{proof}
    	Let $\delta>0$ denote the injectivity radius at $x_0$. As follows from our construction of $Y=\overline{Nx_0}$ (see \Cref{rmk:minimal set on beta level set}) we know that $a_{\delta/4}x_0$ and $a_{-\delta/4}x_0$ are not contained in $Y$, hence there exist small neighborhoods $Q_\pm$ around $a_{\pm\delta/4}x_0$, respectively, which are disjoint from $Y$. Let $0<\eta<\delta/4$ be sufficiently small so that
    	\[ a_{\pm \delta/4}U_{(-\eta,\eta)}x_0 \subset Q_\pm. \]
    	
    	Choose any open interval $(c,d) \subseteq (-\eta/2,\eta/2)$ where both $A_+u_cx_0$ and $A_+u_dx_0$ have infinite slack, e.g.~by choosing these rays to have lifts in $\T^1\Hplane$ terminating outside the limit set. We define
    	\[ \Psi:=A_{(-\delta/4,\delta/4)}U_{(c,d)}x_0. \]
    	Recall that $NAU$ corresponds to the open Bruhat cell in $\PSL_2(\R)$, which means in particular that the multiplication map $ N \times A \times U \to NAU $ is a diffeomorphism (see e.g.~\cite[Lemma 6.44]{knappLieGroupsIntroduction2002} and \S2.3 in \cite{farreMinimizingLaminationsRegular2023a} for more details).
    	By our choice of constants and the fact that the parameterizations of $U$, $N$ and $A$ are of unit speed, we conclude that $N_{(-\delta/4,\delta/4)}\Psi \subset B_\delta^{\T^1\Sigma_s}(x_0)$ hence showing the multiplication map $N_{(-\delta/4,\delta/4)} \times A_{(-\delta/4,\delta/4)}\times U_{(c,d)} \to N_{(-\delta/4,\delta/4)}\Psi$ is a homeomorphism. 
    	
    	Injectivity of the map above implies, in particular, that $\Psi$ is a section for the $N$-flow.
    	Moreover, the fact that $N_{(-R,R)}\Psi$ is open in $\T^1\Sigma_s$ for all $R<\delta/4$ implies property (1) of the statement. 
    	
    	Now note that
    	\[ \partial \Psi = A_{[-\delta/4,\delta/4]}u_cx_0 \cup A_{[-\delta/4,\delta/4]}u_dx_0 \cup a_{-\delta/4}U_{[c,d]}x_0 \cup a_{\delta/4}U_{[c,d]}x_0. \]
    	By our choice of $c$ and $d$ we know that $Y \cap (Au_cx_0 \cup Au_dx_0)=\emptyset$. In addition,  $Y \cap a_{\pm\delta/4}U_{[c,d]}x_0\subset Y \cap (Q_-\cup Q_+) = \emptyset$. 
        We have thus concluded property (2), and the proof of the lemma.
    \end{proof}
    
	\subsection{Acknowledgments} OL would like to thank FD for the warm hospitality at Rennes University during January 2024. OL would also like to thank Zemer Kosloff, Fran\c{c}ois Ledrappier, George Peterzil, Ariel Rapaport, and Omri Sarig for various helpful discussions.
    
    JF acknowledges the support of the Institut Henri Poincaré (UAR 839 CNRS-Sorbonne Université) and LabEx CARMIN (ANR-10-LABX-59-01) and  DFG – Project-ID 281071066 – TRR 191. OL acknowledges the support of ISF grant 957/25. YM acknowledges the support of NSF grant DMS-2005328.
	
\bibliography{Refrences}{}
\bibliographystyle{amsalpha.bst}

\end{document}